\documentclass[amsfonts,12pt]{amsart}

\usepackage{epsfig}
\usepackage{amssymb}
\usepackage{graphicx,amsmath,amssymb,pgf}
\newcommand*{\parallelogramm}{%
  \rlap{\rotatebox{-30}{\rule[.05ex]{.4pt}{.77em}}}%
  \kern.04em%
  \rlap{\kern.36em\raisebox{0.649519052835em}{\rule{.6em}{.4pt}}}%
  \rule{.6em}{.4pt}\kern-.04em%
  \rotatebox{-30}{\rule[.05ex]{.4pt}{.77em}}}

\newtheorem{thm}{Theorem}[section]
\newtheorem{deff}[thm]{Definition}
\newtheorem{lem}[thm]{Lemma}

\newtheorem{rem}[thm]{Remark}
\newtheorem{prp}[thm]{Proposition}

\newtheorem{prob}[thm]{Problem}

\newcommand{\conv}{{\mathrm{conv}}\,}

\newcommand{\R}{{\Bbb R}}

\newcommand{\Z}{{\Bbb Z}}

\begin{document}

 \title[Aleksandrov projection problem for convex lattice sets]{Aleksandrov projection problem for convex lattice sets}

 \author{Ning~Zhang}

 \address{Ning Zhang, Department of Mathematical and Statistical Sciences, University of Alberta, Edmonton AB, T6G 2G1, Canada}
 \email{nzhang2@ualberta.ca}

\thanks{Author is partially supported by a grant from NSERC}

\begin{abstract} 
Let $K$ and $L$ be origin-symmetric convex integer polytopes in $\R^n$. We study a discrete analogue of the Aleksandrov projection problem. If for every $u\in \Z^n$, the sets $(K\cap \Z^n)|u^\perp$ and $(L\cap \Z^n)|u^\perp$ have the same number of points, is then $K=L$? We give a positive answer to this problem in $\Z^2$ under an additional hypothesis that $(2K\cap \Z^2)|u^\perp$ and $(2L\cap \Z^2)|u^\perp$ have the same number of points.
\end{abstract}

\maketitle

\section{Introduction}
%Geometric tomography studies the unique determination of convex (or star) bodies from their sections or projections. There are an abundance of results and Aleksandrov's projection problem is one of most famous among them. That is, an origin-symmetric convex body can be uniquely determined by the areas of its projections; see \cite{Ga}.

%Discrete tomography is an important branch of geometric tomography, which considers problems of finite subsets of the integer lattice \cite{HK}. Gardner, Gronchi, and Zong initially take results from geometric tomography and try to adapt it to discrete settings; in particular, they asked whether it is possible to determine convex lattice sets from the size of their projections (see \cite{GGZ}).

Let $K$ be a convex body in $\R^n$, i.e. a compact convex set with nonempty interior. We say that $K$ is origin-symmetric if $K=-K$, where $tK:=\{tx\in \R^n: x\in K\}, t\in \R$. In 1937, Aleksandrov proved the following result \cite{Ga}:
\begin{thm}
Let $K,L\subset \R^n$ be origin-symmetric convex bodies. If $$\textnormal{vol}_{n-1}(K|u^\perp)=\textnormal{vol}_{n-1}(L|u^\perp)$$ for every $u\in S^{n-1}$, then $K$=$L$.
\end{thm}
Here $u^\perp:= \{x\in \R^n:\langle x,u\rangle=0\}$. 

Gardner, Gronchi, and Zong suggested a discrete version of the Aleksandrov projection problem (see \cite{GGZ}). We say $A$ is a convex lattice set if $\conv(A)\cap \Z^n= A$, where $\conv(A)$ is the convex hull of $A$.
\begin{prob}\label{p1.1}
Let $A,B\subset \Z^n$ be origin-symmetric convex lattice sets. If $|A|u^\perp|=|B|u^\perp|$ for every $u\in\mathbb{Z}^n$, is it true that $A=B$?
\end{prob}
Here, $|A|u^\perp|$ is the cardinality of $A|u^\perp$.  Since the convex hull of a convex lattice set is a convex integer polytope, i.e. a polytope all of whose vertices are in $\Z^n$, it would be convenient to restate the problem as follows. Let $K,L\subset \R^n$ be origin-symmetric convex integer polytopes. If $|(K\cap \Z^n)|u^\perp|=|(L\cap \Z^n)|u^\perp|$ for every $u\in\mathbb{Z}^n$, is it true that $K=L$?

In \cite{GGZ}, the authors gave a negative answer to Problem \ref{p1.1} in $\Z^2$. However, it is not known whether there are other counterexamples. Zhou \cite{Zhou} and Xiong \cite{Xiong} showed that these counterexamples are unique in some special classes.
For higher dimensions, this problem is still open. Some work on related problems has been done in \cite{RYZ}. Since the answer is negative in dimension $2$, Gardner, Gronchi, and Zong asked if it is possible to impose reasonable additional conditions to make the answer affirmative. In this paper, we obtain a positive answer to Problem \ref{p1.1} in $\mathbb Z^2$ under an additional hypothesis.

Before we state the theorem, some definition should be introduced (see \cite{Ga} and \cite{Z}). Let $K$ be a convex body in $\R^n$. The support function of $K$ in the direction $u$ is
$$
h_K(u):=\sup \{\langle u,x\rangle: x\in K\}.
$$

The width function of $K$ in the direction $u$ is
$$
w_K(u):=h_K(u)+h_K(-u).
$$

If $K$ is a convex integer polytope, then we denote
$$
D_1K:=\{u\in Z^n: \exists x_1,x_2\in K\cap \Z^n, u\parallel x_1x_2\}.
$$ 

For a directed segment $u$ with the initial point $(p_1,\dots,p_n)\in \Z^n$ and the end point $(q_1,\dots,q_n)\in \Z^n$, let
$$\hat{u}:=(\frac{q_1-p_1}{d},\dots, \frac{q_n-p_n}{d})$$
denote the primitive vector in the direction $u$, where $d=\gcd(q_1-p_1,\dots,q_n-p_n)$.

We will need the well-known Pick's theorem (see \cite{Z}). Let $K\subset \R^2$ be a convex integer polygon. Then
$$
\mbox{vol}_2 (K)=|K\cap \Z^2|-\frac 12|\partial K\cap \Z^2|-1,
$$
where $\partial K$ is the boundary of $K$.

We are now ready to state our main result.
\begin{thm}\label{t1.2}
Let $K,L\subset \R^2$ be origin-symmetric convex integer polygons. If $$|(K\cap \Z^2)|u^\perp|=|(L\cap \Z^2)|u^\perp|$$ and $$|(2K\cap \Z^2)|u^\perp|=|(2L\cap \Z^2)|u^\perp|$$ for all $u\in \Z^2$, then $K=L$.
\end{thm}
\begin{rem}\label{r1.4}
It will be clear from the proof that we do not need projections in all directions, only in directions parallel to the edges of $K$ and $L$, and one more direction $\xi\in \Z^2\backslash(D_1K\cup D_1 L)$.
\end{rem}
\section{proof of Theorem \ref{t1.2}}
\begin{thm}\label{t3.2}
Let $K$ be an origin-symmetric convex integer polygon in $\R^2$ with edges $\{e_i\}_{i=1}^{2n}$, where $e_i$ and $e_{n+i}$ are symmetric with respect to the origin. Then
$$
|(K\cap \Z^2)|{e_i^\perp}|=|\hat{e}_i|w_K(e_i^\perp)+1,\ \mbox{for }1\leq i\leq n,
$$
where $|\hat{e}_i|$ is the length of the primitive vector parallel to $e_i$. Here and below, $w_K(u^\perp)$ means the width in the direction perpendicular to $u$.
\end{thm}

We will first prove the theorem in a simple case.
\begin{lem}\label{l3.3}
Let $K\subset \R^2$ be a parallelogramm with edges $\{e_i\}_{1\leq i\leq 4}$, where $e_1\parallel e_3$ and $e_2\parallel e_4$. Then
$$
|(K\cap\Z^2)|{e_i^\perp}|=|\hat{e}_i|w_K(e_i^\perp)+1,\ \mbox{for } i=1,2.
$$
\end{lem}

\begin{proof}
Consider the point lattice $\Lambda$ generated by $\hat{e}_1$ and $\hat{e}_2$ and the quotient map $\pi: \R^2\to \R^2/\Lambda$. Set $l(e_1)$ to be the line passing through the origin and parallel to $e_1$.
If $x\in K\cap \Lambda$, then 
$$
|(x+l(e_1))\cap (K\cap\Z^2)|=|e_1\cap \Z^2|.
$$ 
If $x\in (K\cap\Z^2)\backslash \Lambda$, then $\pi((x+l(e_1))\cap (K\cap\Z^2))$ contains only one point; otherwise, $x\in \Lambda$. Thus,
$$
|(x+l(e_1))\cap (K\cap\Z^2)|=|e_1\cap \Z^2|-1.
$$
One can see that,
$$
|(K\cap \Lambda)|e_1^\perp|=|e_2\cap \Z^2|.
$$
Furthermore when projecting $(K\cap\Z^2)\backslash \Lambda$ onto $e_1^\perp$,  each point in the projection has $|e_1\cap \Z^2|-1$ preimages. Thus,
$$
|((K\cap\Z^2)\backslash \Lambda)|e_1{\tiny }^\perp|=\frac{|(K\cap\Z^2)|-|e_1\cap \Z^2||e_2\cap \Z^2|}{|e_1\cap \Z^2|-1};
$$
hence,
\begin{eqnarray*}
|(K\cap\Z^2)|{e_1^\perp}|&=&\frac{|(K\cap\Z^2)|-|e_1\cap \Z^2||e_2\cap \Z^2|}{|e_1\cap \Z^2|-1}+|e_2\cap \Z^2|\\
&=& \frac{|(K\cap\Z^2)|-|e_2\cap \Z^2|}{|e_1\cap \Z^2|-1}\\
&=& \frac{\mbox{vol}_2(K)+|e_1\cap \Z^2|-1}{|e_1\cap \Z^2|-1}\ (\mbox{by Pick's theorem})\\
&=& \frac{|\hat{e}_1|(|e_1\cap \Z^2|-1) w_K(e_1^\perp)+|e_1\cap \Z^2|-1}{|e_1\cap \Z^2|-1}\\
&=& |\hat{e}_1|w_K(e_1^\perp)+1.
\end{eqnarray*}
\end{proof}

\begin{proof}[Proof of Theorem \ref{t3.2}]
Without loss of generality, we only need to compute $|(K\cap \Z^2)|{e_1^\perp}|$. Create a convex lattice set with convex hull being a parallelogramm with edges $e_1$ and $e_{n+1}$, denoted by $\parallelogramm$. Note that, for any $x\in (K\cap \Z^2)\backslash \parallelogramm$, $x+l(e_1)\cap \parallelogramm\neq \emptyset$. Thus, there exists $m\in \Z$, such that $x\in \parallelogramm+me_1$, which implies $x-me_1\in \parallelogramm\cap \Z^2$. Therefore, by Lemma \ref{l3.3}
$$
|(K\cap\Z^2)|{e_1^\perp}|=|(\parallelogramm\cap\Z^2)|{e_1^\perp}|= |\hat{e}_1|w_K(e_1^\perp)+1.
$$
\end{proof}

Theorem \ref{t3.2} implies that if $K$ and $L$ have parallel edges, then there is a uniqueness in Problem \ref{p1.1}.
\begin{lem}\label{l3.5}
Let $K$ be an origin-symmetric convex integer polygon in $\mathbb{Z}^2$. Let $u\in D_1K$. If $2(|(K\cap \Z^2)|{u^\perp}|-1)=|(2K\cap \Z^2)|{u^\perp}|-1$,
then 
$$
|(K\cap \Z^2)|{u^\perp}|=|\hat{u}| w_{K}(u^\perp)+1.
$$
\end{lem}
\begin{proof}
Let $u\in D_1K$. If $u$ is parallel to one of the edges of $K$, then, by Theorem \ref{t3.2}, $|(K\cap \Z^2)|{u^\perp}|=|\hat{u}| w_{K}(u^\perp)+1$; if not, consider the pair of points $(x_1,x_2)\in\{(x,y)\in K\times K: xy\parallel \hat{u}\}$ such that
$$
\mbox{dist}(O,\overline{x_1x_2})=\max_{\{(x,y)\in K\times K: xy\parallel \hat{u}\}}\mbox{dist}(O,\overline{xy}).
$$
Here, we denoted by $\mbox{dist}(O,A)=\inf_{x\in A}\|x-O\|_2$, the distance between $O$ and a set $A$. The set $\{(x,y)\in K\times K: xy\parallel \hat{u}\}$ is not empty, since $u\in D_1K$.

Thus, the lines passing through $x_1,x_2$ and $-x_1,-x_2$ divide $\mathbb{R}^2$ into three parts $E_1, E_2$, and $E_3$, where $O\in E_2$ and $E_1,E_3$ are reflections of each other with respect to $O$. 

Note that, $E_2\cap K\cap \Z^2$ is a convex lattice set and $x_1,x_2,-x_1,-x_2$ lie on two parallel edges of $E_2\cap K$. (Here, $E_2\cap K$ can be a segment.) Then, by Theorem \ref{t3.2}, we have 
$$
|(E_2\cap K\cap \Z^2)|{u^\perp}|=|\hat{u}| w_{E_2\cap K}(u^\perp)+1
$$
and set $|(E_1\cap K\cap \Z^2)|{u^\perp}|=|(E_3\cap K\cap \Z^2)|{u^\perp}|=m$. We have
$$
|(K\cap \Z^2)|{u^\perp}|=2m+|\hat{u}| w_{E_2\cap K}(u^\perp)-1.
$$

On the other hand, $|(2E_2 \cap2K\cap \Z^2)|{u^\perp}|=2|\hat{u}| w_{E_2\cap K}(u^\perp)+1$. Moreover, a line $l$ parallel to $u$ divides $2E_1\cap 2K$ into two parts of equal width in the direction perpendicular to $u$, denoted by $E_{11}\cap 2K$ and $E_{12}\cap 2K$, where $\mbox{dist}(O,E_{11})> \mbox{dist}(O,E_{12})$. 

Note that there exists a pair of points $y_1,y_2\in l\cap 2K\cap \Z^2$. To see this, pick a point $z$ from $E_1\cap K\cap\Z^2$ such that $w_{[-z,z]}(u^\perp)=w_{K}(u^\perp)$, where $[-z,z]$ is the segment connecting $-z$ and $z$. Then $2z, 2x_1,2x_2\in 2K\cap \Z^2$ implies $y_1=z+x_1,y_2=z+x_2\in 2K\cap l$.

Now we obtain $E_{11}\cap 2K\supset E_1\cap K+z$. To see this, assume $E_1\cap K=\{x\in \R^2: \langle x, v_i\rangle\leq a_i\}$ with $x_1,x_2\in \{x\in \R^2: \langle x, v_1\rangle= a_1\}$, then $\langle z,v_1\rangle=a_1-w_{E_1\cap K}(v_1)$ and $\langle u,v_1\rangle=0$. Thus for any $x\in E_1\cap K$, $\langle x+z, v_i\rangle\leq 2a_i$ and $\langle x+z,v_1\rangle\leq2a_1 -w_{E_1\cap K}(v_1)$, implying $x+z\in 2(E_1\cap 2K)\cap E_{11}=E_{11}\cap 2K$.

Since $E_{12}\cap 2K$  contains a parallelogramm $\parallelogramm$ with vertices $2x_1,x_1+x_2,y_1,y_2$, we have
$$
|(E_{11}\cap 2K\cap \Z^2)|u^\perp|\geq |(E_1\cap K\cap \Z^2)|u^\perp|=m
$$
and
$$
|(E_{12}\cap 2K\cap \Z^2)|u^\perp|=|\hat{u}| w_{\parallelogramm}(u^\perp)+1=|\hat{u}| w_{E_1\cap K}(u^\perp)+1.
$$
Hence,
$$
|(2E_1\cap 2K\cap \Z^2)|{u^\perp}|\geq m+ |\hat{u}| w_{E_1\cap K}(u^\perp).
$$
Therefore,
$$
|(2K\cap \Z^2)|{u^\perp}|\geq 2(m+ |\hat{u}|  w_{E_1\cap K}(u^\perp))+2 |\hat{u}|  w_{E_2\cap K}(u^\perp)-1.
$$
By the assumption, we have
\begin{eqnarray*}
&&
2(2m+ |\hat{u}|  w_{E_2\cap K}(u^\perp)-2)=2(|(K\cap \Z^2 )|{u^\perp}|-1)\\
&&=|(2K\cap \Z^2)|{u^\perp}|-1\geq 2(m+ |\hat{u}|  w_{E_1\cap K}(u^\perp))+2 |\hat{u}| w_{E_2\cap K}(u^\perp)-1,
\end{eqnarray*}
which implies
$$
m\geq  |\hat{u}|  w_{E_1\cap K}(u^\perp)+1.
$$

On the other hand, $m\leq |\hat{u}| w_{E_1\cap K}(u^\perp)+1$, by constructing a large parallelogramm containing $E_1\cap K$, that has two edges parallel to $u$ and whose width perpendicular to $u$ is $w_{E_1\cap K}(u^\perp)$; thus,
$$
m=  |\hat{u}| w_{E_1\cap K}(u^\perp)+1.
$$
The conclusion follows.
\end{proof}

\begin{deff}\label{d3.6}
Let $\mathcal{K}^n$ be the collection of all origin-symmetric convex bodies in $\mathbb{R}^n$. Define an operator $\Cup:\mathcal{K}^n\times\mathcal{K}^n\to \mathcal{K}^n$, satisfying
$$
A\Cup B:= \conv(A\cup B).
$$
\end{deff}

One can easily prove the following properties.
\begin{prp}\label{p3.8}
Let $A,B\in \mathcal{K}^n$. Then
$$
h_{A\Cup B}(u)=\max\{h_A(u),h_B(u)\}\ \mbox{ and }\  w_{A\Cup B}(u)=\max\{w_A(u),w_B(u)\}.
$$
\end{prp}

\begin{lem}\label{l3.10}
Let $K$ and $L$ be origin-symmetric convex polygons in $\mathbb{R}^2$. If $w_K(u^\perp)=w_L(u^\perp)$ for all $u\in E_K\cup E_L$, then $K=L$. Here, $E_K$ is the collection of all directions parallel to the edges of $K$.
\end{lem}

\begin{proof}
Clearly, $K\subseteq K\Cup L$ and $w_K(u^\perp)=w_{K\Cup L}(u^\perp)$, for all $u\in E_K$. Assume $K\subsetneqq K\Cup L$. Then there exists a point $\{x\}\in K\Cup L$, but $\{x\}\notin K$. Thus we can find a direction $\eta\in E_K$, such that, $2\langle x,\eta\rangle=\langle x-(-x),\eta\rangle>w_K(\eta^\perp)$, implying $w_K(\eta^\perp)< w_{[-x,x]\Cup K}(\eta^\perp)$. On the other hand, since $[-x,x]\Cup K\subseteq K\Cup L$, we have
$$
w_{[-x,x]\Cup K}(\eta^\perp)\leq w_{K\Cup L}(\eta^\perp)=w_K(\eta^\perp),
$$
Contradiction.
\end{proof}

\begin{proof}[Proof of Theorem \ref{t1.2}]
Here, we use the weaker condition mentioned in Remark \ref{r1.4}. Note that, $|(K\cap \Z^2)|u^\perp|<|K\cap \Z^2|$, if $u\in D_1K$; but $|(K\cap \Z^2)|u^\perp|=|K\cap \Z^2|$, if $u\in \Z^2\backslash D_1K$. For any $u\in E_K$, we have $u\in D_1L$; otherwise, $$|(L\cap\Z^2)|u^\perp|=|L\cap \Z^2|=|(L\cap\Z^2)|\xi^\perp|=|(K\cap\Z^2)|\xi^\perp|=|K\cap\Z^2|>|(K\cap\Z^2)|u^\perp|$$ for some $\xi\in \Z^2\backslash (D_1K\cup D_1L)$. Then by Lemma \ref{l3.5}, we have
$$
|(K\cap \Z^2)|{u^\perp}|=|\hat{u}|w_K(u^\perp)+1\ \mbox{ and }\ |(2K\cap \Z^2)|{u^\perp}|=2|\hat{u}|w_K(u^\perp)+1.
$$
By the assumption,
\begin{eqnarray*}
&&
|(2L\cap \Z^2)|{u^\perp}|-1=|(2K\cap \Z^2)|{u^\perp}|-1=2|\hat{u}|w_K(u^\perp)\\
&&=2(|(K\cap \Z^2)|{u^\perp}|-1)=2(|(L\cap \Z^2)|{u^\perp}|-1).
\end{eqnarray*}
Applying Lemma \ref{l3.5},
$$
|(L\cap \Z^2)|{u^\perp}|=|\hat{u}|w_L(u^\perp)+1=|(K\cap \Z^2)|{u^\perp}|=|\hat{u}|w_K(u^\perp)+1.
$$
Therefore,
$$
w_L(u^\perp)=w_K(u^\perp),
$$
for any $u\in E_K$. Similarly, we can show $w_L(u^\perp)=w_K(u^\perp)$, for any $u\in E_L$. Then the conclusion follows from Lemma \ref{l3.10}.
\end{proof}
\section{Acknowledgement}

I would like to express my gratitude to my supervisor, Dr. Vladyslav Yaskin, for fruitful discussions.

\end{document}